\documentclass[12pt,english,a4paper]{smfart}

\usepackage[T1]{fontenc}
\usepackage{lmodern}
\usepackage{smfthm}
\usepackage[headings]{fullpage}
\usepackage{amssymb}

\newcommand{\Qp}{\mathbf{Q}_p}
\newcommand{\Cp}{\mathbf{C}_p}
\newcommand{\Kalg}{\overline{K}}
\newcommand{\ZZ}{\mathbf{Z}}
\newcommand{\dbf}{\mathbf{d}}
\newcommand{\OO}{\mathcal{O}}
\newcommand{\MM}{\mathfrak{m}}
\newcommand{\wt}{\tilde{w}}
\newcommand{\val}{\operatorname{val}}
\newcommand{\Res}{\operatorname{Res}}
\newcommand{\Disc}{\operatorname{Disc}}
\newcommand{\ResPol}{\operatorname{ResPol}}
\newcommand{\wideg}{\operatorname{wideg}}
\newcommand{\dcroc}[1]{[\![ #1 ]\!]}

\renewcommand{\geq}{\geqslant}
\renewcommand{\leq}{\leqslant} 

\author{Laurent Berger}
\address{UMPA de l'ENS de Lyon \\
UMR 5669 du CNRS}
\email{laurent.berger@ens-lyon.fr}
\urladdr{perso.ens-lyon.fr/laurent.berger/}

\date{\today}

\title[Weierstrass preparation and resultants]{The Weierstrass preparation theorem and resultants of $p$-adic power series}

\begin{document}

\begin{abstract}
We define the resultant of two power series with coefficients in the ring of integers of a $p$-adic field. In order to do this, we prove a universal version of the Weierstrass preparation theorem.
\end{abstract}



\maketitle

\tableofcontents

\setlength{\baselineskip}{18pt}

\section*{Introduction}

Given two polynomials $P$ and $Q$ with coefficients in a field $K$, the resultant $\Res(P,Q)$ allows us to determine whether $P$ and $Q$ have a common root in $\Kalg$. The resultant is a polynomial function of the coefficients of $P$ and $Q$, and $\Res(P,Q)=0$ if and only if $P$ and $Q$ have a common root.

In this article, we consider a similar question for $p$-adic power series. Let $K$ be a finite extension of $\Qp$, or more generally a finite totally ramified extension of $W(k)[1/p]$ where $k$ is a perfect field of characteristic $p$. Let $\OO_K$ denote the integers of $K$, let $\MM_K$ be the maximal ideal of $\OO_K$, let $k$ be the residue field of $\OO_K$, and let $\pi$ be a uniformizer of $\OO_K$. Let $\Cp$ be the completion of an algebraic closure $\Kalg$ of $K$, so that $\MM_{\Cp}$ is the $p$-adic open unit disk. A power series $f(X) =f_0+f_1X+ \cdots \in \OO_K\dcroc{X}$ defines a bounded holomorphic function on $\MM_{\Cp}$, and may have roots in this domain. Given two such power series, we would like to know whether they have a common root. The Weierstrass degree $\wideg(f)$ of $f$ is the smallest integer $n$ such that $f_n \in \OO_K^\times$, or $+\infty$ if there is no such integer. If $\wideg(f)=n$ is finite, then $f$ has precisely $n$ roots (counting multiplicities) in $\MM_{\Cp}$. Our main result is the following.

\begin{enonce*}{Theorem A}
\label{thmA}
For all $n \geq 1$, there exists a power series
\[ \Res_n(\{F_i\}_{i\geq 0},\{G_i\}_{i\geq 0}) \in \ZZ[F_n,F_n^{-1},\{F_k\}_{k \geq n+1},\{G_k\}_{k \geq 0}] \dcroc { F_0, \hdots, F_{n-1} } \] such that for all power series $f(X)$, $g(X) \in \OO_K\dcroc{X}$ with $\wideg(f)=n$, we have 
\[ \prod_{\substack{z \in \MM_{\Cp} \\  f(z)=0}} g(z) = \Res_n(\{f_i\}_{i\geq 0},\{g_i\}_{i\geq 0}). \]
In particular, $\Res_n(\{f_i\}_{i\geq 0},\{g_i\}_{i\geq 0})=0$ iff $f$ and $g$ have a common root in $\MM_{\Cp}$.
\end{enonce*}

Note that if $\wideg(f)=n$, then $f_0,\hdots,f_{n-1} \in \MM_K$ so that $\Res_n(\{f_i\}_{i\geq 0},\{g_i\}_{i\geq 0})$ converges. The classical resultant of two polynomials $P$ and $Q$ can be defined using either the product $\prod_{P(z)=0} Q(z)$ or the determinant of the Sylvester matrix. Both approaches give the same formula, after a suitable normalization. Theorem A follows the first approach. Is it possible to view $\Res_n$ as the (generalized) determinant of an operator on some $p$-adic Banach space?

The main technical tool for proving theorem A is the Weierstrass preparation theorem. We use a version due to O'Malley (see \cite{OM72}) which allows us to prove the following universal Weierstrass preparation theorem. Recall that if $R$ is a ring and $I$ is an ideal of $R$, a polynomial is said to be distinguished if it is monic and all its non-leading coefficients are in $I$. If $n \geq 1$, let $R_n = \ZZ[F_n,F_n^{-1},\{F_k\}_{k \geq n+1}] \dcroc { F_0,\hdots,F_{n-1} }$ and let $I_n$ be the ideal of $R_n$ generated by $F_0,\hdots,F_{n-1}$.

\begin{enonce*}{Theorem B}
\label{thmB}
We can write $F(X) = \sum_{i \geq 0} F_i X^i \in R_n \dcroc{X}$ as $F(X)=P(X)U(X)$ where $U(X) \in R_n \dcroc{X}^\times$ and $P(X) = X^n + P_{n-1} X^{n-1} + \cdots + P_0 \in R_n[X]$ is a distinguished polynomial for $I_n$. In addition, $P$ and $U$ are uniquely determined by $F$.
\end{enonce*}

Theorem B provides a universal Weierstrass preparation theorem, and the existence part of the classical versions follows by specializing. In particular, Theorem B shows how the coefficients of $p$ and $u$ depend on those of $f$ when we write a power series $f(X) \in \OO_K \dcroc{X}$ as the product of a distinguished polynomial $p$ and a unit $u$.

In \S\ref{lubinsec}, we give an application of our results to the iteration of power series in characteristic $p$. We show that such a power series admits a lift to characteristic zero satisfying certain properties, which strengthens a construction of Lubin (see \cite{JL95}).

We finish this article with a sketch of an analogue of our constructions that singles out the roots of a power series in a sphere, instead of in an open disk.

\section{A universal Weierstrass preparation theorem}
\label{weiersec}

The classical Weierstrass preparation theorem over $\OO_K$ (see for instance \cite{AC}, chapter VII, \S3, no 8) says that if $f(X) \in \OO_K \dcroc{X}$ and $\wideg(f) = n$, there exists a distinguished (for the ideal $\MM_K$) polynomial $p(X)$ of degree $n$ and a unit $u(X) \in \OO_K \dcroc{X}^\times$ such that $f=pu$. In addition, $p$ and $u$ are uniquely determined by $f$. The coefficients of $p$ and $u$ depend on those of $f$. In order to make this dependence more explicit, we use the following strengthening of the Weierstrass preparation theorem, which is theorem 2.10 of \cite{OM72}.

\begin{theo}
\label{weier}
Let $R$ be a ring, and take $f(X) = f_0+f_1X+ \cdots \in R \dcroc{X}$. Suppose that $f_n \in R^\times$ and that $R$ is separated and complete for the $(f_0,\hdots,f_{n-1})$-adic topology. 

There exists a distinguished (for $(f_0,\hdots,f_{n-1})$) polynomial $p(X)$ of degree $n$ and $u(X) \in R \dcroc{X}^\times$ such that $f=pu$. In addition, $p$ and $u$ are uniquely determined by $f$.
\end{theo}

Although we don't need this in the remainder of this article, we point out the following corollary of theorem \ref{weier}. Note that some even more general versions of the Weierstrass preparation theorem can be found, see for instance \cite{JE14}.

\begin{coro}
\label{Iadic}
Let $R$ be a ring and let $J$ be an ideal of $R$ such that $R$ is separated and complete for the $J$-adic topology. Take $f(X) = f_0+f_1X+ \cdots \in R \dcroc{X}$. Suppose that $f_n \in R^\times$ and that $f_0,\hdots,f_{n-1} \in J$. 

There exists a distinguished (for $J$) polynomial $p(X)$ of degree $n$ and $u(X) \in R \dcroc{X}^\times$ such that $f=pu$. In addition, $p$ and $u$ are uniquely determined by $f$.
\end{coro}

\begin{proof}
This follows from theorem \ref{weier}, and the following assertion (\cite{SP}, Tag 00M9, lemma 10.95.8): if $I \subset J$ are two ideals of a ring $R$, with $I$ finitely generated, and if $R$ is separated and complete for the $J$-adic topology, then $R$ is separated and complete for the $I$-adic topology.
\end{proof}

If $n \geq 1$ is fixed,  we can consider the variables $\{F_i\}_{i \geq 0}$ and we define 
\[ R_n = \ZZ[F_n,F_n^{-1},\{F_k\}_{k \geq n+1}] \dcroc { F_0,\hdots,F_{n-1} }. \]
The ring $R_n$ is separated and complete for the $(F_0,\hdots,F_{n-1})$-adic topology, and the following result (Theorem B) is an immediate consequence of theorem \ref{weier}.

\begin{theo}
\label{weierexpl}
We can write $F(X) = \sum_{i \geq 0} F_i X^i \in R_n \dcroc{X}$ as $F(X)=P(X)U(X)$ where $U(X) \in R_n \dcroc{X}^\times$ and $P(X) = X^n + P_{n-1} X^{n-1} + \cdots + P_0 \in R_n[X]$ is a distinguished polynomial. In addition, $P$ and $U$ are uniquely determined by $F$.
\end{theo}

\begin{exem}
\label{bgw}
We give an explicit formula for $P(X)$ in theorem \ref{weierexpl} when $n=1$.

If $n=1$, then $P(X)=X+P_0$ and $P_0 \in R_1 = \ZZ[F_1,F_1^{-1},\{F_k\}_{k \geq 2}] \dcroc {F_0} $ is given by the following formula (proposition 2.2 of \cite{BGW14}):
\[ P_0 = \sum_{n \geq 0} F_0^{n+1} \sum_{j=0}^n (-F_1)^{-n-j} \sum_{\substack{i_1+i_2 + \cdots + i_n = j \\ i_1+2i_2 + \cdots + n i_n = n}} \frac{(n+j)!}{(n+1)! i_1! i_2! \cdots i_n!} F_2^{i_1} F_3^{i_2} \cdots F_{n+1}^{i_n}. \]
(We have $(n+j)!/(n+1)! i_1! i_2! \cdots i_n! \in \ZZ$ if $i_1+i_2 + \cdots + i_n = j$ and $i_1+2i_2 + \cdots + n i_n = n$; indeed, $(n+j)!/(n+1)! i_1! i_2! \cdots i_n!$ becomes a multinomial coefficient and hence an integer if we replace either $n+1$ by $n$ or $i_k$ by $i_k-1$ for some $k$.
If $\ell$ is a prime number, then it cannot divide both $n+1$ and all of the $i_k$. Hence $(n+j)!/(n+1)! i_1! i_2! \cdots i_n!$ is a rational number that is $\ell$-integral for every prime number $\ell$, and is therefore an integer).
\end{exem}

\section{Resultants and discriminants of $p$-adic power series}
\label{resulsec}

By the theory of Newton polygons, a distinguished polynomial $p(X) = X^n + p_{n-1} X^{n-1} + \cdots + p_0 \in \OO_K[X]$ of degree $n$ has precisely $n$ roots in $\MM_{\Cp}$ (counting multiplicities). Let $p(X)$ be such a polynomial. If $g(X) = \sum_{i \geq 0} g_i X^i \in \OO_K\dcroc{X}$, we can consider $\prod_{p(z)=0} g(z)$.

\begin{prop}
\label{respoly}
There exists a power series \[ \ResPol_n(P_0,\hdots,P_{n-1}, \{G_k\}_{k \geq 0}) \in \ZZ[\{G_k\}_{k \geq 0}] \dcroc{P_0,\hdots,P_{n-1}} \] such that for all
$g(X) = \sum_{i \geq 0} g_i X^i \in \OO_K\dcroc{X}$ and all distinguished polynomial $p(X) = X^n + p_{n-1} X^{n-1} + \cdots + p_0 \in \OO_K[X]$ of degree $n$, we have
\[ \prod_{p(z)=0} g(z) = \ResPol_n(p_0,\hdots,p_{n-1},\{g_k\}_{k \geq 0}). \]
\end{prop}

\begin{proof}
Let $Z_1,\hdots,Z_n$ denote $n$ variables. For each $n$-uple $\dbf = (d_1,\hdots,d_n) \in \ZZ_{\geq 0}$ with $d_1 \leq \cdots \leq d_n$, let $Z_\dbf = \sum_{(e_1,\hdots,e_n)} Z_1^{e_1} \cdots Z_n^{e_n}$ where $(e_1,\hdots,e_n)$ ranges over all distinct permutations of $(d_1,\hdots,d_n)$.
There exists polynomials $S_{\dbf} \in \ZZ[\{G_k\}_{k \geq 0}]$ for each $\dbf$, such that
\[ \prod_{i=1}^n \sum_{k \geq 0} G_k Z_i^k = \sum_{\dbf} S_\dbf(\{G_k\}_{k \geq 0}) Z_\dbf. \]

If we write $\prod_{i=1}^n (X-Z_i) = X^n + P_{n-1} X^{n-1} + \cdots + P_0$, then by the fundamental theorem of symmetric polynomials, each $Z_\dbf$ belongs to $\ZZ[P_0,\hdots,P_{n-1}]$. We set 
\[ \ResPol_n = \sum_{\dbf} S_\dbf(\{G_k\}_{k \geq 0}) Z_\dbf \in \ZZ[\{G_k\}_{k \geq 0}] \dcroc{P_0,\hdots,P_{n-1}}. \] 
Note that the total degree of $Z_\dbf$ is $d_1+\cdots+d_n$ so that the degree of $Z_\dbf$ as an element of $\ZZ[P_0,\hdots,P_{n-1}]$ is at least $(d_1+\cdots+d_n)/n$. Therefore, the above sum converges for the $(P_0,\hdots,P_{n-1})$-adic topology. The proposition follows by specializing.
\end{proof}

We can now prove Theorem A.

\begin{theo}
\label{resul}
There exists a power series \[ \Res_n(\{F_i\}_{i\geq 0},\{G_i\}_{i\geq 0}) \in \ZZ[F_n,F_n^{-1},\{F_k\}_{k \geq n+1},\{G_k\}_{k \geq 0}] \dcroc { F_0, \hdots, F_{n-1} } \] such that for all power series $f(X)$, $g(X) \in \OO_K\dcroc{X}$ with $\wideg(f)=n$, we have 
\[ \prod_{\substack{z \in \MM_{\Cp} \\ f(z)=0}} g(z) = \Res_n(\{f_i\}_{i\geq 0},\{g_i\}_{i\geq 0}). \]
\end{theo}

\begin{proof}
By theorem \ref{weierexpl}, we can write $F(X) = \sum_{i \geq 0} F_i X^i$ as $F(X)=P(X)U(X)$ with $P(X)=X^n+P_{n-1} X^{n-1} + \cdots +P_0$, where each $P_i$ belongs to the ideal $(F_0,\hdots,F_{n-1})$ of the ring $\ZZ[F_n,F_n^{-1},\{F_k\}_{k \geq n+1}] \dcroc { F_0,\hdots,F_{n-1} }$. The result follows from proposition \ref{respoly}, by setting $\Res_n= \ResPol_n(P_0,\hdots,P_{n-1}, \{G_k\}_{k \geq 0})$.
\end{proof}

If $f, g \in \OO_K\dcroc{X}$ and $\wideg(f)=n$, we write $\Res_n(f,g)$ instead of $\Res_n(\{f_i\}_{i\geq 0},\{g_i\}_{i\geq 0})$ in order to lighten the notation.

\begin{rema}
\label{resprod}
We have $\Res_n(f,gh)=\Res_n(f,g)\Res_n(f,h)$.
\end{rema}

\begin{defi}
\label{defdisc}
We define $\Disc_n$ to be the power series \[ \Disc_n(\{F_i\}_{i\geq 0}) = \Res_n(F,F') \in \ZZ[F_n,F_n^{-1},\{F_k\}_{k \geq n+1}] \dcroc { F_0, \hdots, F_{n-1} }, \] and likewise write $\Disc_n(f)$ instead of $\Disc_n(\{f_i\}_{i\geq 0})$.
\end{defi}

By theorem \ref{resul}, a power series $f(X) \in \OO_K\dcroc{X}$ with $\wideg(f)=n$ has only simple roots in $\MM_{\Cp}$ if and only if $\Disc_n(f) \neq 0$.

\begin{prop}
\label{opensimple}
The set of elements of $\OO_K \dcroc{X}$ having only simple roots in $\MM_{\Cp}$ is open in the $p$-adic topology.
\end{prop}

\begin{proof}
Take $f(X) \in \OO_K \dcroc{X}$ having only simple roots in $\MM_{\Cp}$. We can divide $f$ by an appropriate power of $\pi$ and assume that $\wideg(f)$ is finite. Let $n = \wideg(f)$ and $v = \val_\pi \Disc_n(f)$. 

The fact that $\Disc_n$ belongs to $\ZZ[F_n,F_n^{-1},\{F_k\}_{k \geq n+1}] \dcroc { F_0, \hdots, F_{n-1} }$ implies that for every $h(X) \in \OO_K \dcroc{X}$, we have $\val_\pi \Disc_n(f+\pi^{v+1} h)=v$, so that if $g(X) \in \OO_K \dcroc{X}$ is such that $\val_\pi(f-g) \geq v+1$, then $g(X)$  has only simple roots in $\MM_{\Cp}$.
\end{proof}

Note that the set of elements of $\OO_K \dcroc{X}$ having only simple roots in $\MM_{\Cp}$ is also dense in the $p$-adic topology. If $f=pu$ with $p$ distinguished having multiple roots, then $p$ can be approached by distinguished polynomials having only simple roots. Indeed, the (usual) discriminant of $p(X)=X^n+p_{n-1} X^{n-1} + \cdots +p_0$ is a polynomial $\Delta(p_0,\hdots,p_{n-1})$ and its zero set is closed with empty interior.

\section{Lubin's proof of Sen's theorem on iteration of power series}
\label{lubinsec}

In this section, we give an application of the above constructions. In his paper \cite{JL95}, Lubin gives a short and very nice proof of Sen's theorem on iteration of power series. We start by recalling Sen's theorem and Lubin's argument. Recall that $k$ is the residue field of $\OO_K$. If $w(X) = X + \sum_{i \geq 2} w_i X^i \in k \dcroc{X}$, let $i(w)=m-1$ where $m$ is the smallest integer $\geq 2$ such that $w_m \neq 0$ (or $+\infty$ if there is no such integer). For $n \geq 0$, let $i_n(w) = i(w^{\circ p^n})$. Sen's theorem (theorem 1 of \cite{S69}) says that $i_{n-1}(w) \equiv i_n(w) \bmod{p^n}$ for all $n \geq 1$ (where the congruence holds automatically if one side is $+\infty$). 

Lubin's argument is to show that for each $n \geq 0$ such that $i_n(w) \neq +\infty$, there exists a finite extension $L$ of $K$ and a power series $f_n(X) \in X \cdot \OO_L \dcroc{X}$ such that the image of $f_n(X)$ in $k_L\dcroc{X}$ is $w(X)$ and such that all the roots of $f_n^{\circ p^n}(X)-X$ in $\MM_{\Cp}$ are simple. We then have $i_n(w)-i_{n-1}(w) = \wideg( (f_n^{\circ p^n}(X)-X)/(f_n^{\circ p^{n-1}}(X)-X) )$, so that $i_n(w)-i_{n-1}(w)$ is the number of points of $\MM_{\Cp}$ whose orbit under $f_n$ is of cardinality $p^n$. This number is clearly divisible by $p^n$, which implies Sen's theorem.

Using our methods, we can improve Lubin's result. We prove that there is one lift $f$ of $w$ that works for all $n$, and has coefficients in $\OO_K$.

\begin{theo}
\label{goodlift}
Take $w(X) = X + \sum_{i \geq 2} w_i X^i \in k \dcroc{X}$ and let $N \subset \ZZ_{\geq 0}$ be the set of $n$ such that $i_n(w)$ is finite. There exists $f(X) \in X \cdot \OO_K \dcroc{X}$ whose image in $k \dcroc{X}$ is $w(X)$ and such that for all $n \in N$, the roots of $f^{\circ p^n}(X)-X$ in $\MM_{\Cp}$ are simple.
\end{theo}

\begin{proof}
Let $W$ be the set of $f(X) \in X \cdot \OO_K \dcroc{X}$ whose image in $k \dcroc{X}$ is $w(X)$, and let $W_n$ be the set of elements of $W$ such that the roots of $f^{\circ p^n}(X)-X$ in $\MM_{\Cp}$ are simple. We prove that if $n \in N$, then $W_n$ is open and dense in $W$ for the $p$-adic topology. Since $W$ is a complete metric space, the theorem follows from this assertion and Baire's theorem, which implies that $\cap_{n \in N} W_n$ is dense in $W$ and hence non-empty.

Fix an element $\wt \in W$. We have $W = \{ \wt + h, h \in \pi X \cdot \OO_K \dcroc{X}\}$. If $F(X) = \sum_{j \geq 1} F_j X^j$, write $F^{\circ p^n}(X) - X = \sum_{j \geq 1} F_j^{(n)} X^j$. Take $n \in N$ and let $i=i_n(w)+1$. Let $F(X) = \sum_{j \geq 1} (\wt_j + H_j) X^j$, where $\{H_j\}_{j \geq 1}$ are variables. For all $j \geq 1$, $F_j^{(n)} \in \OO_K[H_1,\hdots,H_j]$. Since $F_i^{(n)}(0) = \wt_i^{(n)} \in \OO_K^\times$,  $F_i^{(n)}$ has an inverse $(F_i^{(n)})^{-1} \in \OO_K \dcroc{H_1,\hdots,H_i}$. If $j \leq i-1$, then $\wt_j^{(n)} \in \MM_K$, and so $F_j^{(n)}$ is in the ideal $(\pi,H_1,\hdots,H_j)$ of $\OO_K[H_1,\hdots,H_j]$. The power series \[ \Disc_i(F^{\circ p^n}(X) - X) \in \ZZ[F_i^{(n)} , (F_i^{(n)})^{-1},\{F_j^{(n)}\}_{j \geq i+1}] \dcroc{F_1^{(n)},\hdots,F_{i-1}^{(n)}} \] therefore gives rise to an element $D_n(\{H_j\}_{j \geq 1}) \in \OO_K[\{H_j\}_{j \geq i+1}] \dcroc{H_1,\hdots,H_i}$.

Let us first show that $W_n$ is open in $W$. If $f = \wt + h \in W_n$, with $h \in \pi X \cdot \OO_K \dcroc{X}$, then $D_n(h) \neq 0$ by definition. If $v = \val_\pi(D_n(h))$ and $g(X) \in X \cdot \OO_K \dcroc{X}$, then $D_n(h+ \pi^{v+1} g) \equiv D_n(h) \bmod{\pi^{v+1}}$ so that $\val_\pi( D_n(h+ \pi^{v+1} g) ) = v$. Hence $f + h' \in W_n$ for all $h' \in \pi X \cdot \OO_K \dcroc{X}$ such that $\val_\pi(h-h') \geq v+1$, and therefore $W_n$ is open in $W$.

We now show that $W_n$ is dense in $W$. If this is not the case, its complement has non-empty interior. Suppose therefore that there exists $f = \wt + h \in W$ and $v \geq 1$ such that $D_n(h + \pi^v g) = 0$ for all $g \in X \cdot \OO_K \dcroc{X}$. We can write $D_n(\{H_j\}_{j \geq 1}) = \sum_{\dbf \in \ZZ_{\geq 0}^{i-1}} P_\dbf( \{H_j\}_{j \geq i+1} ) H_1^{d_1} \cdots H_i^{d_i}$ where $\dbf=(d_1,\hdots,d_i)$ and the $P_\dbf$ are polynomials with coefficients in $\OO_K$. The fact that $D_n(h + \pi^v g) = 0$ for all $g \in X \cdot \OO_K \dcroc{X}$ implies that for all fixed values of $\{g_j\}_{j \geq i+1}$, the corresponding power series in $H_1,\hdots,H_i$ is zero on the  set $(h_1,\hdots,h_i)+\pi^v \OO_K^i$. It is therefore the zero power series. This in turn implies that for each $\dbf$, the polynomial $P_\dbf( \{H_j\}_{j \geq i+1} )$ is zero on the set $\{ (h_j + \pi^v \OO_K) \}_{j \geq i+1}$ and therefore $P_\dbf=0$. 

This implies that $D_n$ is the zero power series, and therefore that for any extension $L/K$ and any power series $f(X) \in X \cdot \OO_L \dcroc{X}$ such that $\wideg(f^{\circ p^n}(X)-X)=i$, the roots of $f^{\circ p^n}(X)-X$ in $\MM_{\Cp}$ are not simple. This is in contradiction with Lubin's result in \cite{JL95} (the aforementioned construction of the power series $f_n$).
\end{proof}

\section{A universal Hensel factorization theorem}

In this section, we sketch an analogue of our constructions that singles out the roots of a power series in a sphere $\{ z \in \Cp, |z|=r\}$ instead of in an open disk as in \S\ref{weiersec} and \S\ref{resulsec}. Let $\OO_K\{X\}$ denote the ring of restricted power series (power series $f(X)=\sum_{n \geq 0} f_n X^n$ with $f_n \in \OO_K$ and $f_n \to 0$ as $n \to +\infty$). An element of $\OO_K\{X\}$ converges on the closed unit disk $\{ z \in \Cp, |z| \leq 1 \}$. We are interested in the roots of $f$ in the unit sphere $\{ z \in \Cp, |z|=1\}$. Take $f(X) = \sum_{n \geq 0} f_n X^n \in \OO_K\{X\}$, one of whose coefficients is in $\OO_K^\times$. Let $\mu_{\min}(f) = \min \{ i \geq 0, f_i \in \OO_K^\times\}$ and let $\mu_{\max}(f) = \max \{ i \geq 0, f_i \in \OO_K^\times\}$. If $n = \mu_{\min}(f)$ and $n+d = \mu_{\max}(f)$, we have the factorization $\overline f = \overline p \cdot \overline u$ in $k [X]$, with 
\[ \overline p (X) =  \overline f_{n+d}^{-1} \cdot (\overline f_n + \overline f_{n+1} X + \cdots + \overline f_{n+d} X^d) \text{ and } \overline u(X) = \overline f_{n+d} \cdot X^n. \]
Hensel's factorization theorem (\cite{AC}, chapter III, \S 4, no 3) implies that there exist $p(X) \in \OO_K[X]$ and $u(X) \in \OO_K\{X\}$ such that $f=pu$, the polynomial $p$ is monic of degree $d$,  $p(0) \in \OO_K^\times$, and $\mu_{\max}(u)=\mu_{\min}(u)$. This analogue of the Weierstrass preparation theorem, along with the theory of Newton polygons,  implies that $f$ has precisely $\mu_{\max}(f)-\mu_{\min}(f)$ roots (counting multiplicities) in the unit sphere.

Let $\{F_i\}_{i \geq 0}$ be variables, take $n,d \geq 0$, and let 
\[ S_{n,d} = \ZZ[\{F_{n+j}\}_{0 \leq j \leq d},F_n^{-1},F_{n+d}^{-1}] \dcroc{ F_0,\hdots,F_{n-1}, \{ F_{n+d+k} \}_{k \geq 1} }. \]
Our definition of a power series ring in infinitely many variables is the ``large'' one (for instance, $\sum_{k \geq 0} F_k$ belongs to $S_{n,d}$), see chapter IV, \S 4 of \cite{ALG}. Let $I_{n,d}$ be the ideal of $S_{n,d}$ generated by $F_0,\hdots,F_{n-1}, \{ F_{n+k} \}_{k \geq 1}$. The following result is a universal Hensel factorization theorem.

\begin{theo}
\label{hensuniv}
We can write $F(X) = \sum_{i \geq 0} F_i X^i \in S_{n,d} \dcroc{X}$ as $F(X)=P(X)U(X)$ where $P(X) \in S_{n,d}[X]$ is monic of degree $d$, $P(0) \in S_{n,d}^\times$, and $U(X) \equiv F_{n+d} X^n \bmod{I_{n,d}}$. In addition, $P$ and $U$ are uniquely determined by $F$.
\end{theo}

\begin{proof}
The ring $S_{n,d}$ is separated and complete for the $I_{n,d}$-adic topology. The polynomials $\overline P (X) = F_{n+d}^{-1} \cdot (F_n + F_{n+1} X + \cdots + F_{n+d} X^d)$ and $\overline U (X) = F_{n+d} \cdot X^n$ generate the unit ideal in $S_{n,d}/I_{n,d} [X]$, since $F_n,F_{n+d} \in S_{n,d}^\times$. Indeed, a descending induction on $n-1 \geq m \geq 0$ shows that $X^m \in (\overline P, \overline U)$ by considering $X^m \overline P$.
 
The theorem therefore results from Hensel's factorization theorem (see \cite{AC}, chapter III, \S 4, no 3, and the discussion at the beginning of no 5 of ibid.).
\end{proof}

\begin{coro}
\label{slopeuniv}
Given $F(X) \in \OO_K[X]$, one of whose coefficients is in $\OO_K^\times$, there are universal formulas, depending only on $\mu_{\min}(F)$, $\mu_{\max}(F)$ and $\deg(F)$, for the coefficients of the slope $0$ factor of $F$ in its slope factorization, in terms of the coefficients of $F$.
\end{coro}

\begin{proof}
Let $F = F_0 F_{\neq 0}$ be the factorization of $F$ as the product of a monic polynomial of slope $0$ and of a polynomial of slopes $\neq 0$. The polynomial $F_{\neq 0}$ has no roots in the unit sphere, so that if we view $F$ as an element of $\OO_K\{X\}$, then $P=F_0$ and $U=F_{\neq 0}$.
\end{proof}

Theorem \ref{hensuniv} can also be used, as in \S\ref{resulsec}, to produce resultant power series $\Res_{n,d}$, that will detect whether two restricted power series $f$ and $g$, with $\mu_{\min}(f)=n$ and $\mu_{\max}(f)=n+d$, have roots in common in the unit sphere.

\providecommand{\og}{``}
\providecommand{\fg}{''}

\end{document}